\documentclass[12pt]{amsart}
\usepackage{graphicx}

\theoremstyle{plain} \newtheorem{thm}{Theorem}[section]
\newtheorem{cor}[thm]{Corollary}
\newtheorem{lemma}[thm]{Lemma}

\theoremstyle{definition} 
\newtheorem{defin}[thm]{Definition}

\usepackage{graphicx}
\DeclareGraphicsExtensions{.pdf,.jpg,.png}

\input{xy}
\xyoption{all}
\begin{document}

\title{On convergence in the subpower Higson corona of metric spaces}
\author{Jacek Kucab}
\author{Mykhailo Zarichnyi}

\address{Faculty of Mathematics and Natural Sciences,
University of Rzesz\'ow, Rejtana 16 A, 35-310 Rze\-sz\'ow, Poland}
\email{jacek.kucab@wp.pl}
\email{zarichnyi@yahoo.com}

\begin{abstract}
The subpower Higson corona of a proper metric space is defined in \cite{KZ}. We prove that, unlikely to the Higson corona, the closure of a $\sigma$-compact subset of the subpower Higson corona of a proper unbounded metric space does not necessarily coincide with its  Stone-\v{C}ech corona.

\end{abstract}
\maketitle
\section{Introduction}
The notions of Higson compactification and Higson corona play an important role in coarse geometry and, in particular, in the asymptotic dimension theory (see, e.g., \cite{R,DKU, BD}).

The sublinear Higson corona $\nu_L(X)$ was introduced and investigated in \cite{DS} in connection with the asymptotic Assouad-Nagata dimension theory. The subpower Higson corona of a proper metric space is defined in \cite{KZ}. This is a counterpart of the sublinear Higson corona. In some sense, the subpower Higson corona is an intermediate corona between the sublinear Higson corona and the Higson corona.

J. Keesling proved that the closure of a $\sigma$-compact subset in the Higson corona of a proper metric space is homeomorphic to its Stone-\v Cech compactification \cite[Theorem 1]{K}. In particular, the Higson corona does not contain non-trivial convergent sequences. This is not the case for the sublinear corona as it is proved in \cite{DS} that the sublinear corona  $\nu_L(X\times\mathbb R_+)$ contains a topological copy of $\nu_L(X)\times[0,1]$. 

In the first version of this note the authors erroneously claimed that the subpower corona behaves similarly to the remainder of the Stone-\v Cech compactification, i.e., the  sublinear Higson corona of a noncompact proper metric space does not contain nontrivial convergent sequences. The authors express their gratitude to Roman Pol and Yutaka Iwamoto for indicating the gap in the proof. 

\section{Preliminaries}
In the sequel, $(X,d)$ is assumed to be an unbounded, proper metric space with the basepoint $x_0$. Recall that a metric space is said to be proper if every its bounded and closed subset is compact.

Let $\vert x\vert=d(x,x_0)$ for $x\in X$.

The following notion is more general than that introduced in \cite{KZ}.

\begin{defin}
A function $p\colon \mathbb{R}_+\to\mathbb{R}_+$ is called \textit{asymptotically subpower} if for every $\alpha>0$ there exists $t_0>0$ such that $p(t)< t^\alpha$ for all $t>t_0$.
\end{defin}

\begin{defin}
A bounded and continous function $f:X\to\mathbb{R}$ is called \textit{Higson subpower} if for every asymptotically subpower function $p$ we have $\mathrm{lim}_{\vert x\vert\to\infty}\mathrm{diam}(f(B_{p(\vert x\vert)}(x)))=0$, more precisely if for any asymptotically subpower function $p$ and $\epsilon>0$ there exists $r_0>0$ that $\mathrm{diam}(f(B_{p(\vert x\vert)}(x)))<\epsilon$ for all $x\in X$ such that $\vert x\vert>r_0$.
\end{defin}

It is easy to verify that Higson subpower functions form a closed subalgebra in the algebra of all continous and bounded functions on a proper metric space $X$. This algebra contains the constant functions and separates points and closed sets. The compactification of $X$ that corresponds to this subalgebra is called the Higson subpower compactification and is denoted by $h_P X$. The remainder of this compactification $h_P X\setminus X$ is denoted by $\nu_P X$ and is called the subpower Higson corona of $X$.

A subset
$D$
of a space
$X$
is called
$C*$-embedded if any continuous bounded function on $D$ can be continuously extended over $X$.

\section{Results}

\begin{lemma}\label{3.4.2}

For any subsets $A,B$ of a space $X$ such that $\nu_P X\cap\overline{A}\cap\overline{B}=\emptyset$ there exist constants $\alpha, r_0>0$ such that $\mathrm{max}\lbrace d(x,A),d(x,B)\rbrace\geq \vert x\vert^\alpha$ for every $x\in X$ such that $\vert x\vert\geq r_0$.

\end{lemma}

\begin{proof}
Let $A,B\subset X$ and let $\nu_P X\cap\overline{A}\cap\overline{B}=\emptyset$. Suppose, that the condition does not hold, i. e. for any $\alpha, r_0>0$ there exists $x\in X$, such that $\vert x\vert\geq r_0$, but $\mathrm{max}\lbrace d(x,A),d(x,B)\rbrace< \vert x\vert^\alpha$. Therefore there exist a sequence $(x_n)$ of elements of $X$ with the property, that for all $n\in\mathbb{N}$ it is $\vert x_n\vert\geq 2^{n^2}$ and $\mathrm{max}\lbrace d(x_n,A),d(x_n,B)\rbrace< \vert x_n\vert^\frac{1}{2n}$. Thus, there exist sequences $(a_n)$ and $(b_n)$ of elements of $A$ and $B$ respectively, such that $d(x_n,a_n)<\vert x_n\vert^\frac{1}{2n}$ and $d(x_n,b_n)<\vert x_n\vert^\frac{1}{2n}$. Let us define continuous function $F:h_P X\to \mathbb{R}$ such that $F_{\vert \overline{A}\cap \nu_P X}=1$ and $F_{\vert \overline{B}\cap \nu_P X}=0$. Existence of such a function is guaranteed by normality of the space $h_P X$ and fact that sets $\overline{A}\cap \nu_P X$ and $\overline{B}\cap \nu_P X$ are closed and disjoint. Since $F$ is continuous, it is an extension of some Higson subpower function $f\colon X\to \mathbb{R}$. Moreover there exist open neighborhoods $U$ and $V$ of $\overline{A}\cap \nu_P X$ and $\overline{B}\cap \nu_P X$ respectively, which contains every but finite elements of sequences $(a_n)$ and $(b_n)$ respectively, such that $f_{\vert U}>3/4$ and $f_{\vert V}<1/4$. Indeed, every infinite sequence in a compact Hausdorff space has at least one cluster point and for the sequences $(a_n)$ (respectively  $(b_n)$) all the cluster points can belong only to $\overline{A}\cap \nu_P X$ and (respectively $\overline{B}\cap\nu_P X$). No point of $X$ can be a cluster point of these sequences, because unboundedness of $(\vert x_n\vert)$ easily implies unboundedness of $(\vert a_n\vert)$ and $(\vert b_n\vert)$.

Now let us consider the function $p\colon \mathbb{R}_+\to \mathbb{R}_+$ defined as follows: $p(t)=1$ for $t<\vert x_1\vert$ and $p(t)=\vert x_n\vert^{1/n}$ for $t\in[\vert x_n\vert,\vert x_{n+1}\vert)$, $n=1,2,\dots$. The function $p$ is asymptotically subpower, as for any $\alpha>0$ there exists $n\in \mathbb{N}$, such that $1/n<\alpha$, and therefore taking $r_0=\vert x_n\vert\geq 1$ for $t>r_0$ with the property that $t\in[\vert x_m\vert,\vert x_{m+1}\vert)$ the condition $p(t)=\vert x_m\vert^{1/m}\leq t^{1/m}\leq t^{1/n}<t^\alpha$ holds.

Moreover, for any $n\in\mathbb{N}$ we have $a_n, b_n\in B_{p(\vert x_n\vert)}(x_n)$, as $\vert x_n\vert^\frac{1}{2n}<\vert x_n\vert^\frac{1}{n}=p(\vert x_n\vert)$ and so $\mathrm{diam}(B_{p(\vert x_n\vert)}(x_n))>1/2$ for every but finite elements of $(x_n)$.

Therefore $\mathrm{lim}_{\vert x_n\vert\to \infty}\mathrm{diam}(B_{p(\vert x_n\vert)}(x_n))\neq 0$. Contradiction.
\end{proof}

The following statement is a counterpart of the corresponding property of the sublinear corona \cite{DS}.

\begin{lemma}\label{3.4.3} For every subsets $E_1,E_2$ of a space $X$ the fact that there exist $r_0,\alpha>0$ such that $\max\lbrace\rho(x,E_1),\rho(x,E_2)\rbrace\geq\vert x\vert ^\alpha$ for all $x\in X$ with $\vert x\vert \geq r_0$ implies $$\nu_P X \cap \overline{E_1}\cap\overline{E_2}=\emptyset.$$
\end{lemma}

\begin{proof} Let $E_1,E_2\subset X$ and let $r_0$ and $\alpha$ satisfy the assumptions. Let $F_i=E_i\setminus B_{r_0+r_0^\alpha}(x_0)$ for $i\in\lbrace 1,2\rbrace$. let also $f\colon X\to \mathbb{R}$ be defined by the formula $f(x)=\rho(x,F_1)+\rho(x,F_2)$, $x\in X$. Remark that $f(x)\geq \vert x\vert ^\alpha$ for all $x$ with $\vert x\vert\geq r_0$, and also $f(x)\geq r_0^\alpha\geq \vert x\vert ^\alpha$ for all $x$ with $\vert x\vert\leq r_0$. Thus, for every $x\in X$ we have $f(x)\geq \vert x\vert ^\alpha$ and $f(x)>0$. For $i\in\lbrace 1,2\rbrace$, define a function $g_i\colon X\to \mathbb{R}$, $g_i(x)=\rho(x,F_i)/f(x)$. Let $p$ be an arbitrary asymptotically subpower function and let $\epsilon>0$. Let $\tilde{r_0}>(3/\epsilon)^{2/\alpha}$ be large enough so that for every $x\in X$ with $\vert x\vert>\tilde{r_0}$ we have $p(\vert x\vert)<\vert x\vert^{\alpha/2}$ and let $y\in B_{p(\vert x\vert)}(x)$. Then for $i\in\lbrace 1,2\rbrace$ the function $g_i$ is continuous, bounded and $$\vert g_i(y)-g_i(x)\vert = \left\vert \dfrac{\rho(y,F_i)}{f(y)}-\dfrac{\rho(y,F_i)}{f(x)}+\dfrac{\rho(y,F_i)}{f(x)}-\dfrac{\rho(x,F_i)}{f(x)}\right\vert\leq$$$$\rho(y,F_i)\left\vert \dfrac{1}{f(y)}-\dfrac{1}{f(x)}\right\vert+\dfrac{\vert \rho(y,F_i)-\rho(x,F_i) \vert}{f(x)}\leq $$$$\dfrac{\rho(y,F_i)}{f(x)f(y)}\vert f(x)-f(y)\vert+\dfrac{\rho(y,x)}{f(x)}\leq\dfrac{\vert f(x)-f(y)\vert}{f(x)}+\dfrac{\rho(x,y)}{f(x)}\leq$$$$\dfrac{2\rho(y,x)}{f(x)}+\dfrac{\rho(x,y)}{f(x)}\leq 3\dfrac{\rho(x,y)}{\vert x\vert ^\alpha}.$$

Then for every $x\in X$ with $\vert x\vert> \tilde{r_0}$ we have $\vert g_i(y)-g_i(x)\vert<\epsilon$, and thus $g_i$ is a subpower Higson function for $i\in\lbrace 1,2\rbrace$. Let $\bar{g_i}\colon h_P X\to \mathbb{R}$ be the unique  extension of $g_i$ onto the subpower Higson compactification, $i\in\lbrace 1,2\rbrace$. Since $g_1(x)+g_2(x)=1$ for every $x\in X$, we see that $\bar{g_1}(x)+\bar{g_2}(x)=1$  for every  $x\in h_P X$. Moreover, $\overline{F_i}\subset \bar{g_i}^{-1}(\lbrace 0\rbrace)$ and $\nu_P X\cap\overline{F_i} = \nu_P X\cap\overline{E_i}$ for $i\in\lbrace 1,2\rbrace$. Therefore $$\nu_P X\cap \overline{E_1}\cap\overline{E_2}=\nu_P X\cap \overline{F_1}\cap\overline{F_2}\subset\nu_P X\cap \bar{g_1}^{-1}(\lbrace 0\rbrace)\cap\bar{g_2}^{-1}(\lbrace 0\rbrace)=\emptyset.$$

\end{proof}

The proof of the following result is suggested by Roman Pol.

\begin{thm} There exists a proper unbounded metric space whose subpower corona contains  a $\sigma$-compact subset which is not  $C^\ast$-embedded.
\end{thm}

\begin{proof} Let $\mathbb{N}^*=\mathbb{N}\setminus\lbrace 1\rbrace$ and let $$Y=\lbrace (y_1,y_2)\in\mathbb{R}^2\mid y_1\geq 0\, ,\, y_2\in [0,\sqrt{y_1}\,]\rbrace.$$ Let $y_0=(0,0)$. We endow $Y$ with the max-metric $d$. Note that for $y=(y_1,y_2)\in Y$ such that $\vert y\vert\geq 1$ we have $\vert y\vert=y_1$. For $i\in \mathbb{N}^*$ let $$Y_i=\lbrace (y_1,y_1^{1/i})\mid y_1\geq 2 \rbrace$$ and $$K_i=\overline{Y_i}\cap\nu_P Y=\nu_P Y_i.$$ Let also $K=\bigcup_{i\in\mathbb{N}^*}K_i$, $K^\prime=\bigcup_{i\in 2\mathbb{N}+1}K_i$ and $K^{\prime\prime}=\bigcup_{i\in 2\mathbb{N}}K_i$.

Remark that $K=K^\prime\cup K^{\prime\prime}$ is a  $\sigma$-compact subspace of  $\nu_P Y$ and consider a function $f\colon K\to [0,1]$ such that $f_{\vert K^\prime}\equiv 1$ and $f_{\vert K^{\prime\prime}}\equiv 0$. Clearly, $f$ is bounded. Let us verify that  $f$ is well defined and continuous. To this end, let us show that for every $i\in\mathbb{N}^*$ there exists a set $W\subset Y$ such that $Y_i\subset W$, $\bigcup_{j\in\mathbb{N}^*\setminus\lbrace i\rbrace}Y_j\subset Y\setminus W$ and there exists $\alpha,r_0>0$ such that for every   $y=(y_1,y_2)\in Y$ with $\vert y\vert\geq r_0$ we have $$\mathrm{max}\lbrace d(y,Y_i),d(y,Y\setminus W)\rbrace\geq \vert y\vert^\alpha.$$

Let $i\in\mathbb{N}^*$ and $$W=\left\{ (y_1,y_2)\in Y\mid y_1\geq 2,\ \ y_2\in \left(y_1^{1/(i+1)},y_1^{1/(i-1)}\right)\right\}.$$ Clearly, $Y_i\subset W$ and $\bigcup_{j\in\mathbb{N}^*\setminus\lbrace i\rbrace}Y_j\subset Y\setminus W$. Without loss of generality, one may consider the points $y\in Y$ with $\vert y\vert>2$ and $y_1^{1/(i+1)}\leq y_2\leq y_1^{1/(i-1)}$. Let $\alpha=\frac{1}{2i(i+1)}$ and $r_0=8^{2i(i+1)}$. Assume that $y_1^{1/(i+1)}\leq y_2\leq y_1^{1/i}$. Then

$$\mathrm{max}\lbrace d(y,Y_i),d(y,Y\setminus W)\rbrace=\mathrm{max}\lbrace d(y,Y_i),\mathrm{min}\lbrace d(y,Y_{i+1}),d(y,Y_{i-1})\rbrace\rbrace.$$

Consider the following two cases.

(1) If $\mathrm{min}\lbrace d(y,Y_{i+1}),d(y,Y_{i-1})\rbrace=d(y,Y_{i+1})$, then

$$\mathrm{max}\lbrace d(y,Y_i),d(y,Y\setminus W)\rbrace\geq\frac{1}{4}\left(y_1^{\frac{1}{i}}-y_1^{\frac{1}{i+1}}\right)\geq\frac{1}{4}\left(y_1^{\frac{1}{i}-\frac{1}{i+1}}-1\right)\geq$$ $$\frac{1}{8}y_1^{\frac{1}{i(i+1)}}\geq y_1^{\frac{1}{2i(i+1)}}=\vert y\vert^\alpha.$$

(2) If $\mathrm{min}\lbrace d(y,Y_{i+1}),d(y,Y_{i-1})\rbrace=d(y,Y_{i-1})$, then

$$\mathrm{max}\lbrace d(y,Y_i),d(y,Y\setminus W)\rbrace\geq\frac{1}{4}\left(y_1^{\frac{1}{i-1}}-y_1^{\frac{1}{i}}\right)\geq\frac{1}{4}\left(y_1^{\frac{1}{i-1}-\frac{1}{i}}-1\right)\geq$$ $$\frac{1}{8}y_1^{\frac{1}{i(i-1)}}\geq y_1^{\frac{1}{2i(i-1)}}\geq\vert y\vert^\alpha.$$

Similar calculations can be made also for the case $y_1^{1/i}\leq y_2\leq y_1^{1/(i-1)}$. Using Lemma \ref{3.4.3} we obtain $$\nu_P Y\cap\overline{Y_i}\cap\overline{Y\setminus W}=\emptyset.$$ Moreover, for every $j\in\mathbb{N}^*$, $j\neq i$, we have $$\overline{Y_j}\cap\nu_P Y\subset \overline{Y\setminus W}\cap\nu_P Y=:F,$$ i.e., $K_j\subset F$. Since  $j\neq i$ is arbitrary, we also have $\bigcup_{j\neq i}K_j\subset F$. At the same time $K_i\cap F=\emptyset$, thus $f$ is well-defined. Therefore, $K_i\subset\nu_P Y\setminus F$, and, since $F$ is closed in $\nu_P Y$, we obtain that $K_i$ is open in $K$. Since $i$ is arbitrary, we obtain that the sets $K^\prime$ and $K^{\prime\prime}$ are simultaneously open and closed in $K$, which implies the continuity of $f$.

Now, let $Y_0=\lbrace (y_1,0)\in Y\mid y_1\geq 2\rbrace$ and let $K_0=\nu_P Y_0$. Show that $K_0\subset\overline{K}$. Assume the contrary. Then there exists $k\in K_0\setminus\overline{K}$. By the normality of $h_P Y$, there exist disjoint open neighborhoods $U_k$ and $V_{\overline{K}}$ of the point $k$ and the set $\overline{K}$ in the space $h_P Y$ respectively. There are neighborhoods $U_k^\prime$ and $V_{\overline{K}}^\prime$ of $k$ and $\overline{K}$ respectively such that $\overline{U_k^\prime}\subset U_k$ and $\overline{V_{\overline{K}}^\prime}\subset V_{\overline{K}}$. Let $U=\overline{U_k^\prime}\cap Y$ and $V=\overline{V_{\overline{K}}^\prime}\cap Y$. Remark that $$\overline{U}\cap\overline{V}\cap\nu_P Y\subset \overline{U_k^\prime}\cap \overline{V_{\overline{K}}^\prime}\cap\nu_P Y\subset U_k\cap V_{\overline{K}}=\emptyset,$$ therefore by Lemma \ref{3.4.2} there exist constants $\tilde{\alpha}, \tilde{r_0}>0$ such that, for every $y\in Y$ with $\vert y\vert>\tilde{r_0}$ we have $$\mathrm{max}\lbrace d(y,U),d(y,V)\rbrace\geq \vert y\vert^{\tilde{\alpha}}.$$

Now, remark that, for every $\alpha>0$, there exist $i_\alpha\in\mathbb{N}^*$ and $r_{i_\alpha}>0$ such that, for every $y\in Y_0$ with $\vert y\vert>r_{i_\alpha}$ we have $$\mathrm{max}\lbrace d(y,Y_{i_\alpha}),d(y,Y_0)\rbrace<\vert y\vert^{\alpha}.$$ Indeed, let $\alpha>0$. Take $i_\alpha\in\mathbb{N}^*$ such that $i_\alpha>\frac{2}{\alpha}$ and $r_{i_\alpha}=2^{i_\alpha}$. At the same time, for $y=(y_1,0)\in Y_0$ such that $\vert y\vert=y_1>r_{i_\alpha}$ we have $$\mathrm{max}\lbrace d(y,Y_{i_\alpha}),d(y,Y_0)\rbrace=d(y,Y_{i_\alpha})\leq 2y_1^{\frac{1}{i_\alpha}}<y_1^{\frac{2}{i_\alpha}}<\vert y\vert ^\alpha.$$

Let $A=Y_0\cap U$. Since $k\in\overline{A}$, the set $A$ is unbounded in the metric space $Y$. Consider the set $A_0=\lbrace (a_n,0)\in A\mid a_n>n$, $n\in \mathbb{N}^*\rbrace$, and, for $i\in\mathbb{N}^*$ the sets $A_i=\lbrace (a_n,a_n^{1/i})\mid (a_n,0)\in A\rbrace.$ Clearly, $A_i\subset Y_i$ for every $i\in\mathbb{N}^*\cup\lbrace 0\rbrace$. Also, $A_0\subset U$. Moreover, since $K_i=\nu_P Y_i$ for every $i\in\mathbb{N}^*$, and the set $V_{\overline{K}}^\prime$ is an open neighborhood of $\overline{K}$ (in particular, for every $K_i$ z $i\in\mathbb{N}^*$), for every $i\in\mathbb{N}^*$ there exist $r_i^\prime>0$ such that for every $y\in Y_i$,  $\vert y\vert>r_i^\prime$, we have $y\in V_{\overline{K}}^\prime\cap Y\subset V$. Indeed, $Y_i\setminus V_{\overline{K}}^\prime=\overline{Y_i}\setminus V_{\overline{K}}^\prime$ is compact and therefore bounded subset of $Y$ for every $i\in\mathbb{N}^*$.

Let $\alpha,r_0>0$. Choose $i_\alpha\in\mathbb{N}^*$ and $r_{i_\alpha}>0$ as is described above. Choose a point $y^\prime=(a,0)\in A_0\subset Y$ so that $\vert y^\prime\vert=a>\mathrm{max}\lbrace r_0,r_{i_\alpha},2r_{i_\alpha}^\prime\rbrace$, where $r_{i_\alpha}^\prime$ is picked for $i_\alpha$ as described above. At the same time $$\mathrm{max}\lbrace d(y^\prime,U),d(y^\prime,V)\rbrace\leq \mathrm{max}\lbrace d(y^\prime,Y_0),d(y^\prime,Y_{i_\alpha})\rbrace< \vert y^\prime\vert^\alpha$$ and we obtain a contradiction.

One can similarly prove that $K_0\subset\overline{K^\prime}$ and $K_0\subset\overline{K^{\prime\prime}}$. By the definition,  $f$ cannot be continuously extended over $\nu_P Y$.

\end{proof}

\begin{cor} The closure of a $\sigma$-compact subset of the subpower corona of a proper unbounded metric space does not necessarily coincide with its  Stone-\v{C}ech corona.\end{cor}


\begin{thebibliography}{99}

\bibitem{BD} G. Bell, A. Dranishnikov, {\it Asymptotic dimension, Topology and its Applications},
Volume 155, Issue 12, 15 June 2008, Pages 1265--1296.
\bibitem{DKU} A. Dranishnikov, J. Keesling, V. V. Uspenskij, {\it On the Higson corona of uniformly contractible spaces}, Topology
37
(1999), no. 4, 791--803.

\bibitem{DS} A.N. Dranishnikov, J. Smith, {\it On asymptotic Assouad-Nagata dimension}, Topology and its Applications, Volume 154, Issue 4,  2007, 934--952.

\bibitem{K} J. Keesling, Subcontinua of the Higson corona, Topology and its Applications
80(1997) 155--160.

\bibitem{KZ} J. Kucab, M. Zarichnyi, \textit{Subpower Higson corona of a metric space}, Algebra and Discrete Mathematics 17(2014) n2, 280--287.


\bibitem{R} John Roe, Lectures in Coarse Geometry, University Lecture Series Vol. 31, American Mathematical Society: Providence, Rhode Island, 2003.


\end{thebibliography}
\end{document}